\documentclass[a4paper]{article}
\usepackage[utf8]{inputenc}	
\usepackage[english]{babel}

\usepackage{amsfonts,amsmath,amssymb,amsthm,amsbsy}
\usepackage{latexsym,bm,enumitem,color}
\usepackage{hyperref}

\newtheorem{teo}{Theorem}[section]
\newtheorem{prop}[teo]{Proposition}
\newtheorem{lemma}[teo]{Lemma}
\newtheorem{cor}[teo]{Corollary}

\theoremstyle{definition}

\newtheorem{eg}{Example}[section]
\newtheorem{rem}{Remark}

\newcommand{\en}{\mathbb{N}}	
\newcommand{\zet}{\mathbb{Z}}	
\newcommand{\qu}{\mathbb{Q}}	

\newcommand{\cFq}{\mathbb{F}_q}		
\newcommand{\cFqk}{\mathbb{F}_{q^k}}	

\newcommand{\zd}{\zet[\delta]}		
\newcommand{\zi}{\zet[i]}			
\newcommand{\zt}{\zet[\tau]}		

\newcommand{\qi}{\qu(i)}			
\newcommand{\qtau}{\qu(\tau)}		
\newcommand{\oo}{\mathcal{O}}		
\newcommand{\og}{\mathcal{O}_g}		

\newcommand{\p}{\mathcal{P}}

\DeclareMathOperator{\nr}{N}		
\DeclareMathOperator{\tr}{Tr}		
\DeclareMathOperator{\ord}{ord}		
\DeclareMathOperator{\End}{End}		

\title{Elliptic Curves with Isomorphic Groups of Points over Finite Field Extensions}
\author{
Clemens Heuberger\thanks{supported by the Austrian Science Fund (FWF): P~24644-N26.} \\ 
Alpen-Adria-Universit\"at Klagenfurt \\ \texttt{clemens.heuberger@aau.at}
\and
Michela Mazzoli\thanks{supported by the Karl Popper Kolleg ``Modeling Simulation Optimization'' funded by the Alpen-Adria-Universit\"at Klagenfurt and by the Carinthian Economic Promotion Fund (KWF).} \\
Alpen-Adria-Universit\"at Klagenfurt \\ \texttt{michela.mazzoli@aau.at}
}
\date{}

\begin{document}

\maketitle

\begin{abstract}
Consider a pair of ordinary elliptic curves $E$ and $E'$ defined over the same finite field $\cFq$. Suppose they have the same number of $\cFq$-rational points, i.e.\ $|E(\cFq)|=|E'(\cFq)|$. In this paper we characterise for which finite field extensions $\cFqk$, $k\geq 1$ (if any) the corresponding groups of $\cFqk$-rational points are isomorphic, i.e.\ $E(\cFqk) \cong E'(\cFqk)$.
\end{abstract}

\section{Introduction} \label{s:intro}

Consider a pair of ordinary elliptic curves $E$ and $E'$ defined over the same finite field $\cFq$, where $q$ is a prime power. Suppose $E$ and $E'$ have the same number of $\cFq$-rational points, i.e.\ $|E(\cFq)|=|E'(\cFq)|$. Equivalently, $E$ and $E'$ have the same characteristic polynomial, the same zeta function, hence the same number of $\cFqk$-rational points for every finite extension $\cFqk$ of $\cFq$, $k\geq 1$. This is equivalent to $E$ and $E'$ being $\cFq$-isogenous -- cf.~\cite[Theorem~1]{Tate66}. 
In this paper we characterise for which field extensions $\cFqk$, if any, the corresponding groups of $\cFqk$-rational points are isomorphic, i.e.\ $E(\cFqk) \cong E'(\cFqk)$.

The question was inspired by an article by C.\ Wittmann~\cite{Wittmann01}; we have summarised his result for the ordinary case in Proposition~\ref{prop:wittmann}. Wittmann's paper answers the question for $k=1$. Our main results are illustrated in Theorem~\ref{thm:mainthm} and Theorem~\ref{proposition:valuation-equivalence}. The first theorem reduces the isomorphism problem to a divisibility question for individual $k$'s. In the second theorem, the latter question is reduced to a simple verification of the multiplicative order of some elements, based only on information for $k=1$. Combining Theorem~\ref{thm:mainthm} and Theorem~\ref{proposition:valuation-equivalence}, we are able to tell for which $k\geq 1$ we have $E(\cFqk) \cong E'(\cFqk)$, given only the order of $E(\cFq)$ and the endomorphism rings of $E$ and $E'$.

\section{Isomorphic groups of $\cFqk$-rational points} \label{s:isom_group}

Let $E$ be an \emph{ordinary} elliptic curves defined over the finite field $\cFq$, where $q$ is a prime power. Let $\tau$ be the Frobenius endomorphism of $E$ relative to $\cFq$, namely $\tau(x,y)=\left(x^q,y^q\right)$. In the ordinary case, the endomorphism algebra $\qu\otimes\End_{\cFq}{(E)}$ of $E$ is equal to $\qtau$ -- cf.~\cite[Theorem~2]{Tate66}.

Since $\qtau$ is an imaginary quadratic field, it can be written as $\qu(\sqrt{m})$ for some square-free integer $m<0$. The ring of integers of $\qu(\sqrt{m})$ is $\zd$ where $\delta = \sqrt{m}$ if $m\equiv 2,3\pmod{4}$, or $\delta = \frac{1+\sqrt{m}}{2}$ if $m\equiv 1\pmod{4}$.

Then we can write $\tau=a+b\delta$ for some $a$, $b\in\zet$. It is well-known that the endomorphism ring of $E$ is an order in $\qtau$, that is $\End{(E)} \cong \oo_g = \zet + g\zd = \zet \oplus g\zet\delta$, where $g$ is the \emph{conductor} of the order $\oo_g$. Since $\zt = \oo_b \subseteq \End{(E)}$, we have $g\mid b$.

\begin{prop}[{\cite[Lemma~3.1]{Wittmann01}}] \label{prop:wittmann}
Let $E/\cFq$ and $E'/\cFq$ be ordinary elliptic curves s.t.\ $|E(\cFq)|=|E'(\cFq)|$. Let $\End{(E)} = \og$ and $\End{(E')}=\oo_{g'}$ be the orders in $\qtau$ of conductor $g$ and $g'$ respectively, let $\tau=a+b\delta$ as above. Then
\[
E(\cFq) \cong E'(\cFq) \quad \Leftrightarrow \quad \gcd{(a-1,b/g)} = \gcd{(a-1,b/g')} \:.
\]
\end{prop}

We note that, since $|E(\cFq)|=q+1-\tr{(\tau)}$, knowing the order of $E(\cFq)$ is equivalent to knowing the Frobenius endomorphism of $E$.
\\

As $E/\cFq$ can always be seen as defined over any field extension $\cFqk$, and the Frobenius endomorphism of $E$ with respect to $\cFqk$ is $\tau^k$, we obtain the following

\begin{cor} \label{cor:k_wittmann}
Let $E$ and $E'$ be as in Proposition~\ref{prop:wittmann}. Fix an integer $k\geq 1$ and write $\tau^k = a_k + b_k \delta$ for suitable $a_k$, $b_k\in\zet$. Then
\[
E(\cFqk) \cong E'(\cFqk) \quad \Leftrightarrow \quad \gcd{(a_k-1,b_k/g)} = \gcd{(a_k-1,b_k/g')} \:.
\]
\end{cor}

\begin{rem} \label{rem:End}
If $E$ and $E'$ are such that $|E(\cFq)|=|E'(\cFq)|$ and $\End{(E)}\cong\End{(E')}$, then it follows easily from Corollary~\ref{cor:k_wittmann} that $E(\cFqk) \cong E'(\cFqk)$ for every $k\geq 1$. So the problem is interesting when $\End{(E)}\not\cong\End{(E')}$ and yet $E(\cFqk) \cong E'(\cFqk)$ for all or some $k$'s. Indeed, there exist examples of curves with isomorphic group of points but different endomorphism rings, as illustrated in Section~\ref{s:examples}.

More precisely, by~\cite[Theorem~4.2]{Waterhouse69}, for every order $\oo$ in the endomorphism algebra of $E$ with $\tau\in\oo$, there exists an elliptic curve $E'$ s.t.\ $E'$ is $\cFq$-isogenous to $E$ and $\End{(E')}\cong\oo$.
\end{rem}

Given a prime number $p$ and an integer $n$, $v_p{(n)}$ is the $p$-adic valuation of $n$, i.e.\ the maximum power of $p$ that divides $n$.

\begin{lemma} \label{lemma:gcd_vp_bis}
Let $x,y,g,h\in\zet$ such that $g \mid y$ and $h\mid y$. Let $\p=\{p \text{ prime} \mid v_p{(g)}\not=v_p{(h)} \}$. For every $p\in\p$, let $s_p=\max{\{ v_p{(g)},v_p{(h)} \}}$. Then the following are equivalent
\begin{equation}\label{eq:gcd_xygh}
\gcd{(x,y/g)} = \gcd{(x,y/h)} \:,
\end{equation}
\begin{equation}\label{eq:gcd_vp_sp}
v_p{(x)} \leq v_p{(y)} - s_p \quad \text{ for all } p\in\p \:.
\end{equation}

\end{lemma}

\begin{proof}
It is enough to prove that~\eqref{eq:gcd_xygh} is equivalent to
\begin{equation}\label{eq:gcd_vp_g_h}
v_p{(x)} \leq v_p{(y/g)} \text{ and } v_p{(x)} \leq v_p{(y/h)} \quad \text{ for all } p\in\p \:.
\end{equation}

Since $v_p{(\gcd{(x,y/g)})}=\min{\{ v_p{(x)}, v_p{(y/g)} \}}$ for any prime number $p$, we have that $\gcd{(x,y/g)} = \gcd{(x,y/h)}$ if and only if
\[ \min{\{ v_p{(x)},v_p{(y/g)} \}}=\min{\{ v_p{(x)},v_p{(y/h)} \}} \]
for every prime $p$ s.t.\ $v_p{(g)}\not=v_p{(h)}$. Since $v_p{(y/g)} \not= v_p{(y/h)}$, the minimum must be $v_p{(x)}$. 

\end{proof}

Combining Corollary~\ref{cor:k_wittmann} and Lemma~\ref{lemma:gcd_vp_bis} yields

\begin{teo} \label{thm:mainthm}
Let $E/\cFq$ and $E'/\cFq$ be ordinary elliptic curves s.t.\ $|E(\cFq)|=|E'(\cFq)|$, let $\tau=a+b\delta$ be their Frobenius endomorphism relative to $\cFq$. Let $\End{(E)} = \og$ and $\End{(E')}=\oo_{g'}$ be the orders in $\qtau$ of conductor $g$ and $g'$ respectively. Let $\p:=\{p \mbox{ prime} \mid v_p{(g)}\not=v_p{(g')} \}$. For every $p\in\p$ define $s_p:=\max{\{ v_p{(g)},v_p{(g')} \}}$.

Fix an integer $k\geq 1$ and let $\tau^k = a_k + b_k \delta$ for some $a_k$, $b_k\in\zet$. Then
\begin{equation}	\label{eq:mainthm_eq}
E(\cFqk) \cong E'(\cFqk) \quad \Leftrightarrow \quad v_p(a_k-1) \leq v_p(b_k) - s_p \quad \text{ for all } p\in\p \:.
\end{equation}
\end{teo}

\vspace{1em}

\begin{rem} \label{rem:p_nmid_a}
Given a prime $p\in\p$ as in Theorem~\ref{thm:mainthm}, we clearly have that $v_p(b)\geq 1$ and $1\leq s_p \leq v_p(b)$. Moreover, $p\nmid a$ because, when the curve is ordinary, $a$ and $b$ are coprime. Suppose this is not the case; then there are $a'$, $b' \in\zet$ s.t.\ $\tau=p(a'+b'\delta)$ and $\nr{(\tau)}=p^2 \nr{(a'+b'\delta)}$, but $\nr{(\tau)}=q=\rho^n$, for some $\rho$ prime and $n\in\en$. Then $\rho$ must be equal to $p$ and therefore $\tr{(\tau)}=\rho\cdot\tr{(a'+b'\delta)} \equiv 0\pmod{\rho}$, so the curve is supersingular.
\end{rem}

In Theorem~\ref{proposition:valuation-equivalence} below we explain how to evaluate the right-hand side of~\eqref{eq:mainthm_eq}, without computing $a_k$, $b_k$ for $k>1$. Before that, we need a couple of technical lemmata.

The following lemma is a slight adaptation of~\cite[Lemma~(5)]{padic-formulas}.

\begin{lemma} \label{lemma:vp(plm_r)}
Let $p$ be a prime, $\ell\geq 1$, $m$ coprime to $p$ and $0<r\leq p^\ell$. Then	
\begin{equation*}
v_p{\left(\binom{p^\ell m}{r}\right)} = \ell - v_p{(r)} \:.
\end{equation*}
\end{lemma}

\begin{proof}
For $1\le j< p^{\ell}$, we clearly have $v_p(j)<\ell$ and $v_p(p^\ell m-j)<\ell$ as well as $p^\ell m-j \equiv -j\pmod{p^\ell}$, thus $v_p(p^\ell m-j)=v_p(j)$.

We conclude that
\begin{equation*}
v_p\biggl(\binom{p^\ell m}{r}\biggr) = \sum_{j=0}^{r-1}v_p(p^\ell m-j)-\sum_{j=1}^r v_p(j)=\ell-v_p(r) \:.
\end{equation*}
\end{proof}

The following lemma seems to be folklore, see for instance \cite{LTElemma}.

\begin{lemma}[Lifting the exponent] \label{lemma:lte}
Let $p$ be a prime, $a$, $b\in\zet$ with $a\equiv b\not\equiv 0\pmod p$ and $k$ be a positive integer. If $p=2$, assume additionally that $a\equiv b\pmod 4$. Then
\begin{equation*}
v_p(a^k-b^k)=v_p(a-b) + v_p(k) \:.
\end{equation*}
\end{lemma}

\begin{teo} \label{proposition:valuation-equivalence}
Let $\delta$ be an algebraic integer of degree $2$, $a$,
  $b$, $g\in\zet$, $p$ a prime number such that $p\mid g \mid b$ and $p\nmid a$. Let $k$ be a
  positive integer and ${(a+b\delta)}^k = a_k+b_k\delta$ for suitable integers
  $a_k$ and $b_k$. Let $e$ be the order of $a$ modulo $p$ (for odd $p$) or the
  order of $a$ modulo $4$ (for even $p$).

  If $(p, v_p(b))=(2, 1)$, assume additionally that $2\nmid k$.

  Then the following assertions
  \begin{equation}\label{eq:original-assertion}
    v_p(a_k-1)> v_p(b_k) - v_p(g)
  \end{equation}
  and
  \begin{equation}\label{eq:second-assertion}
    v_p(a^e-1)- v_p(e) > v_p(b)-v_p(g)\text{ and }e\mid k\text{, or }(p, v_2(k), v_2(b/g))=(2, 0, 0)
  \end{equation}
  are equivalent.
\end{teo}

Note that $v_p(e)=0$ unless $p=2$ and $a\equiv -1\pmod{4}$.

\begin{rem} \label{rem:nastycase}
The case of $p=2$, $v_p(b)=1$ and $2\mid k$ can be reduced to
Theorem~\ref{proposition:valuation-equivalence} by considering
\begin{equation*}
  {(a+b\delta)}^k={\left(a^2+2ab\delta+b^2\delta^2\right)}^{k/2}
\end{equation*}
and noting that $2b$ and $b^2$ are certainly divisible by $4$, so
Theorem~\ref{proposition:valuation-equivalence} can be applied on the 
right-hand side after rewriting $\delta^2$ in terms of $1$ and $\delta$.
\end{rem}

\begin{proof}
  Set $v_p(b)=t$ and $v_p(k)=\ell$. We claim that
  \begin{equation}\label{eq:a_k-b_k-congruence}
    \begin{aligned}
      a_k&\equiv a^k \pmod{p^{\ell+t+1}} \:,\\
      b_k&\equiv ka^{k-1}b \pmod{p^{\ell+t+1}} \:.
    \end{aligned}
  \end{equation}
  For $k=1$, there is nothing to show. Write $k=p^\ell m$ for some integer $m$ with $p\nmid m$. Then
  \begin{equation*}
    a_k+b_k\delta = (a+b\delta)^k= a^{k} + ka^{k-1}b\delta +
    \sum_{r=2}^{k}\binom{k}{r} a^{k-r}b^r\delta^r \:.
  \end{equation*}
  To prove \eqref{eq:a_k-b_k-congruence}, we have to show that the last sum is
  in $p^{\ell+t+1}\zet[\delta]$. It is sufficient to show that
  \begin{equation}\label{eq:divisibile-by-p-power}
    p^{\ell+t+1}\mid\binom{k}{r} b^r
  \end{equation}
  for $2\le r\le k$. This is immediately clear for $tr\ge
  \ell+t+1$ as well as $\ell=0$. So we may restrict ourselves to the case $2\le r\le \ell+1$, in particular $\ell>0$.
  By the Bernoulli inequality, we have $p^\ell\ge 1+\ell(p-1)\ge 1+\ell\ge r$,
  thus we can apply Lemma~\ref{lemma:vp(plm_r)} and obtain
  \begin{equation*}
    v_p\Bigl(\binom{p^\ell m}{r}b^r\Bigr)=\ell+rt-v_p(r).
  \end{equation*}
  In order to prove~\eqref{eq:divisibile-by-p-power} we have to show
  that $\ell+rt-v_p(r)\ge \ell+t+1$ or, equivalently,
  \begin{equation}\label{eq:inequality-valuation}
    t(r-1)\ge v_p(r)+1 \:.
  \end{equation}
  If $v_p(r)=0$, there is nothing to show. Otherwise, we have 
  \begin{equation*}
    r\ge p^{v_p(r)}\ge 1+(p-1)v_p(r)
  \end{equation*}
  by the Bernoulli inequality. For $p\neq 2$, or $p=2$ and $t\ge 2$, this
  implies~\eqref{eq:inequality-valuation} immediately.
  For $(p,t)=(2,1)$ we use the additional assumption that $2\nmid k$, i.e.\ $\ell=0$, for which~\eqref{eq:divisibile-by-p-power} has already been seen to hold.

This concludes the proof of~\eqref{eq:a_k-b_k-congruence}. It follows that
\begin{align}
v_p(b_k) & = \ell+t \:, \label{eq:valuation-b_k} \\
v_p(a_k-a^k) & \geq \ell+t+1\:. \label{eq:valuation-a_k}
\end{align}

  If $v_p(a^k-1)\ge \ell+t+1$, then
  \begin{equation*}
    v_p(a_k-1)=v_p\bigl((a_k-a^k) + (a^k-1)\bigr)\ge \ell+t+1 \:,
  \end{equation*}
  and we have $v_p(a_k-1)> \ell+t=v_p(b_k)>v_p(b_k)-v_p(g)$.
  Otherwise, from~\eqref{eq:valuation-a_k} we have
  \begin{equation*}
    v_p(a_k-1)=v_p\bigl((a_k-a^k) + (a^k-1)\bigr)= v_p(a^k-1) \:.
  \end{equation*}

  We conclude that the assertions~\eqref{eq:original-assertion}
  and
  \begin{equation}\label{eq:14-b}
    v_p(a^k-1)>\ell+t-v_p(g)
  \end{equation}
  are equivalent.

  If $e\mid k$, then 
  \begin{equation}\label{eq:lte-consequence}
    v_p(a^k-1)=v_p(a^{e
      (k/e)}-1)=v_p(a^e-1)+v_p\Bigl(\frac{k}{e}\Bigr)=v_p(a^e-1)+\ell - v_p(e) \:.
  \end{equation}
  by Lemma~\ref{lemma:lte}.

  Consider the case of $p\neq 2$.
  As $\ell+t-v_p(g)\ge 0$ because of $g\mid b$, \eqref{eq:14-b} can
  only hold if $a^k\equiv 1\pmod p$ and thus $e\mid k$. By using
  \eqref{eq:lte-consequence}, we see that \eqref{eq:14-b}
  and \eqref{eq:second-assertion} are equivalent.

  Consider now the case of $p=2$. If $\ell=0$ and $v_2(g)=t$, then~\eqref{eq:14-b}
  holds because $v_2(a_k-1)\ge 1$ is always true. Otherwise,
  $\ell+t-v_p(g)>0$ and $v_p(a^k-1)>\ell+t-v_p(g)>0$ implies that
  $v_2(a^k-1)>1$, thus $e\mid k$. In that case, \eqref{eq:lte-consequence}
  implies the equivalence of \eqref{eq:14-b} and \eqref{eq:second-assertion}.
\end{proof}

\section{Examples} \label{s:examples}

\begin{eg} \label{eg:q=3329,tr=50}
Let $q=3329$ and consider the elliptic curve
\[
E_0/\cFq : y^2 = x^3 + 49x \:.
\]

The Frobenius endomorphism of $E_0$ is equal to $\tau=25+52i$ (so $a=25$ and $b=52=2^2\cdot 13$), $E_0$ is ordinary and its endomorphism algebra is $\qi$. Moreover, the endomorphism ring of $E_0$ is isomorphic to $\zi$, thus its conductor is equal to $1$.

The orders of $\qi$ containing $\tau$ are precisely those whose conductor belongs to $\mathcal{G}=\{1,2,2^2,13,2\cdot 13,2^2\cdot 13\}$. As observed in Remark~\ref{rem:End}, for every $g\in\mathcal{G}$ there is an elliptic curve $E/\cFq$ s.t.\ $E$ is $\cFq$-isogenous to $E_0$ (i.e.\ $|E(\cFq)|=|E_0(\cFq)|$) and $\End{(E)}\cong\og$.

For instance, we have found the following:
\[
\begin{array}{ll}
E_1 : y^2 = x^3 + x + 57		& \quad \End{(E_1)}\cong \oo_{2^2\cdot 13} \\
E_2 : y^2 = x^3 + x + 98		& \quad \End{(E_2)}\cong \oo_{13} \\
E_3 : y^2 = x^3 + x + 378		& \quad \End{(E_3)}\cong \oo_{2\cdot 13} \\
E_4 : y^2 = x^3 + 3x + 1152		& \quad \End{(E_4)}\cong \oo_{2} \\
E_5 : y^2 = x^3 + 30x + 351		& \quad \End{(E_5)}\cong \oo_{2^2} \:.
\end{array}
\]

In this example $\p\subseteq\{2,13\}$. Moreover, we have
\begin{equation*}
v_{2}(b) = 2,\:\: e_{2} = 1,\:\: 1 \leq s_{2} \leq 2 \quad \text{and} \quad
v_{13}(b) = 1,\:\: e_{13} = 2,\:\: s_{13} = 1.
\end{equation*}

Therefore
\begin{align}
v_{2}(a^{e_{2}}-1) - v_{2}(e_{2}) = v_{2}(24) = 3 \quad & \text{and} \quad 0 \leq v_{2}(b)-s_{2} \leq 1 \:, \label{eq:q=3329_p=2} \\
v_{13}(a^{e_{13}}-1) - v_{13}(e_{13}) = v_{13}(25^2-1) = 1 \quad & \text{and} \quad v_{13}(b)-s_{13} = 0 \:. \label{eq:q=3329_p=13}
\end{align}

By Theorem~\ref{proposition:valuation-equivalence} and \eqref{eq:q=3329_p=2} we obtain
\[
v_2(a_k-1)>v_2(b_k)-s_2 \quad \text{ for all } k\geq 1 \:,
\]
thus every time that $2\in\p$, we have $E_i(\cFqk)\not\cong E_j(\cFqk)$, regardless of $k$.

Proposition~\ref{proposition:valuation-equivalence} and \eqref{eq:q=3329_p=13} also yield
\[
v_{13}(a_k-1)>v_{13}(b_k)-s_{13} \quad \text{ for all } k\geq 1 \text{ s.t.\ } 2\mid k  \:,
\]
therefore when $\p=\{13\}$ we have $E_i(\cFqk)\cong E_j(\cFqk)$ if and only if $2\nmid k$.
\\

For instance, consider $E_0: y^2 = x^3 + 49x$ and $E_1 : y^2 = x^3 + x + 57$.

Since $\End{(E_0)}\cong\oo_{1}$ and $\End{(E_1)}\cong\oo_{2^2\cdot 13}$, we have $\p=\{2,13\}$ and
\[
E_0(\cFqk)\not\cong E_1(\cFqk) \quad \text{for all } k\geq 1 \:.
\]

Now consider $E_3 : y^2 = x^3 + x + 378$ and $E_4 : y^2 = x^3 + 3x + 1152$. Then $\End{(E_3)}=\oo_{2\cdot 13}$ and $\End{(E_4)}=\oo_{2}$, so $\p=\{13\}$ and
\[
E_3(\cFqk)\cong E_4(\cFqk) \quad \Leftrightarrow \quad 2\nmid k \:.
\]

In a similar manner we can compare all the other curves. We have summarised the outcome in Table~\ref{table:summary_q=3329_tr=50}.

\begin{table}[!ht]
\begin{center}
\begin{tabular}{*{7}{c|}}
$\cong$	& $E_0$	& $E_1$	& $E_2$	& $E_3$	& $E_4$	& $E_5$ \\
\hline
$E_0$	& --	& none		& $k$ odd	& none		& none		& none \\
$E_1$	& --	& --		& none		& none		& none		& $k$ odd \\
$E_2$	& --	& --		& --		& none		& none		& none \\
$E_3$	& --	& --		& --		& --		& $k$ odd	& none \\
$E_4$	& --	& --		& --		& --		& --		& none \\
$E_5$	& --	& --		& --		& --		& --		& --
\end{tabular}
\caption{Values of $k$ for which the groups of $\cFqk$-rational points are isomorphic.}
\label{table:summary_q=3329_tr=50}
\end{center}
\end{table}

\end{eg}


\begin{eg} \label{eg:q=3329,tr=104}

Let $q=3329$ and consider the following elliptic curves defined over $\cFq$:
\begin{align*}
E_0 & : y^2 = x^3 + 99x \\
E_1 & : y^2 = x^3 + x + 72 \\
E_2 & : y^2 = x^3 + x + 192 \:.
\end{align*}

One can verify that these curves are all ordinary and $\cFq$-isogenous, with $|E_i(\cFq)|=3226$ and $\tau=52+25i$. Their endomorphism algebra is $\qi$, thus $a=52$, $b=25$ and $g\in\{1,5,25\}$. However, the three curves have different endomorphisms rings; more precisely
\begin{equation*}
\End{(E_0)}\cong\zi \:, \quad \End{(E_1)}\cong\zet[25i] \:, \quad \End{(E_2)}\cong\zet[5i] \:.
\end{equation*}

For all the curves above we have that $\p=\{5\}$, $v_5(b)=2$, $1\leq s_5\leq 2$, $e_5=\ord{(a)}\pmod{5}=4$ and
\begin{equation*} 
v_{5}(a^{e_{5}}-1) - v_{5}(e_{5}) = v_{5}(52^4-1) = 1 \:.
\end{equation*}

Let $k$ be a positive integer. By Theorem~\ref{proposition:valuation-equivalence}, if $e_5\nmid k$, i.e.\ $4\nmid k$, then $v_5(a_k-1)\leq v_5(b_k)-s_5$, and therefore $E_i(\cFqk)\cong E_j(\cFqk)$ ($i,j=0,1,2$). If $k$ is a multiple of $4$, the groups of $\cFqk$-rational points are isomorphic if and only if $v_{5}(a^{e_{5}}-1) \leq v_5(b)-s_5$, namely $s_5\leq 1$.

For instance, consider the curves $E_0$ and $E_1$. Their conductors are $1$ and $25$ respectively. Then in this case $s_5=2$ and therefore
\[
E_0(\cFqk)\cong E_1(\cFqk) \quad \Leftrightarrow \quad 4\nmid k \:.
\]

On the other hand, consider $E_0$ and $E_2$. Since the conductor of $E_2$ is equal to $5$, we have $s_5=1$ and then
\[
E_0(\cFqk)\cong E_2(\cFqk) \quad \text{for all } k\geq 1 \:.
\]

Clearly, we also have $E_1(\cFqk)\cong E_2(\cFqk)$ $\Leftrightarrow$ $4\nmid k$.

\end{eg}


\begin{eg} \label{eg:q=1031,tr=-20}

Let $q=1031$ and consider the following elliptic curves defined over $\cFq$:
\begin{align*}
E_1 & : y^2 = x^3 + 982x + 824 \\
E_2 & : y^2 = x^3 + x + 13 \\
E_3 & : y^2 = x^3 + x + 89 \\
E_4 & : y^2 = x^3 + 168x + 48 \:.
\end{align*}

One can verify that these curves are all ordinary and $\cFq$-isogenous, with $\tau=-10+7\sqrt{-19} = -17+14\delta$ where $\delta=\frac{1+\sqrt{-19}}{2}$, since the ring of integers of $\qu(\tau)$ is $\zd=\oo_1$. Thus $a=-17$, $b=14$ and $g\in\{1,2,7,14\}$. The respective endomorphisms rings are
\begin{align*}
\End{(E_1)} & \cong\oo_7 \:, \\
\End{(E_2)} & \cong\oo_1 \:, \\
\End{(E_3)} & \cong\oo_{14} \:, \\
\End{(E_4)} & \cong\oo_2 \:.
\end{align*}

We have that $\p\subseteq\{2,7\}$. Moreover,
\begin{equation*}
v_{2}(b) = 1,\:\: e_{2} = 2,\:\: s_{2}=1 \quad \text{and} \quad
v_{7}(b) = 1,\:\: e_{7} = 3,\:\: s_{7}=1 .
\end{equation*}

Let $k$ be a positive integer. If $k$ is not a multiple of $e_{7}=3$, then $v_{7}(a_k-1)\leq v_{7}(b_k)-s_{7}$. Otherwise,	
\[ v_{7}(a^{e_{7}}-1) - v_{7}(e_{7}) = v_{7}({(-17)}^{3}-1) = 1 > 0 = v_{7}(b)-s_{7} \:. \]

We are in the nasty case of $v_{2}(b)=1$, so in order to apply Theorem~\ref{proposition:valuation-equivalence}, we have to assume that $2\nmid k$. Then $(v_{2}(k),v_{2}(k)-s_2)=(0,0)$ and therefore $v_2(a_k-1)>v_2(b_k)-s_2$ for all odd $k$'s. Suppose now that $2\mid k$ and write
\begin{equation*}
{(a+b\delta)}^k = {(A+B\delta)}^{k/2} = A_{k/2} + B_{k/2}\delta \:,
\end{equation*}		
so that $A_{k/2}=a_k$ and $B_{k/2}=b_k$. Moreover
\[ {(a+b\delta)}^2 = {(-17+14\delta)}^2 = -691 - 280\delta \:, \]
thus $A=-691$ and $B=-280$, so now $v_2(B)=3>1$ and we can apply Theorem~\ref{proposition:valuation-equivalence} to ${(A+B\delta)}^{k/2}$. Let ${e'}_2=\ord{(A)} \pmod{4}$, so ${e'}_2=1$. Then we have
\begin{align*}
v_{2}(A^{{e'}_{2}}-1) - v_{2}({e'}_{2}) = v_{2}(-692) & = 2 \:,\\
v_{2}(B)-s_{2} & = 2 \:.
\end{align*}

This yields $v_{2}(A_{k/2}-1)\leq v_{2}(B_{k/2})-s_{2}$, namely $v_{2}(a_k-1)\leq v_{2}(b_k)-s_{2}$, for all even $k$'s.
\\

We conclude that (for $i,j\in\{1,2,3,4\}$)
\begin{itemize}
	\item $\p=\{2\}$: $E_i(\cFqk)\cong E_j(\cFqk)$ if and only if $2\mid k$ ;
	\item $\p=\{7\}$: $E_i(\cFqk)\cong E_j(\cFqk)$ if and only if $3\nmid k$ ;
	\item $\p=\{2,7\}$: $E_i(\cFqk)\cong E_j(\cFqk)$ if and only if $2\mid k$ and $3\nmid k$ ;
\end{itemize}

These comparisons are displayed in Table~\ref{table:summary_q=1031_tr=-20}.

\begin{table}[!ht]
\begin{center}
\begin{tabular}{*{5}{c|}}
$\cong$		& $E_1$	& $E_2$	& $E_3$	& $E_4$	\\
\hline
$E_1$	& --		& $3\nmid k$	& $2\mid k$		& $2\mid k$ and $3\nmid k$ \\
$E_2$	& --		& --			& $2\mid k$ and $3\nmid k$		& $2\mid k$ \\
$E_3$	& --		& --			& --			& $3\nmid k$ \\
$E_4$	& --		& --			& --			& --
\end{tabular}
\caption{Values of $k$ for which the groups of $\cFqk$-rational points are isomorphic.}
\label{table:summary_q=1031_tr=-20}
\end{center}
\end{table}

\end{eg}


\clearpage

\bibliography{hyperbiblio.bib}
\bibliographystyle{plain}

\end{document}